\documentclass[a4paper,11pt]{article}
\usepackage[utf8]{inputenc}
\usepackage{a4}
\usepackage{amsmath}
\usepackage{amssymb}
\usepackage{amsfonts}
\usepackage{amsthm}
\usepackage{graphicx}
\usepackage{proof}
\usepackage{color}
\usepackage{nicefrac}

  \newtheorem{thm}{Theorem}[section]
 \newtheorem{cor}[thm]{Corollary}
 \newtheorem{prop}[thm]{Proposition}

 \newtheorem{lemma}[thm]{Lemma}
 \newtheorem{rem}[thm]{Remark}

 \newcommand{\lx}{\frac{\log p}{x}}

\title{On the asymptotic behaviour of the quantiles in the gamma distribution}
\author{Henrik Laurberg Pedersen}

\begin{document}

\maketitle

\begin{abstract}
 The asymptotic behaviour of the quantiles in the gamma distribution are investigated as the shape parameter tends to zero. Some remarks about the behaviour at infinity are given.
\end{abstract}

\section{Introduction}
The $p$-quantile $m_p(x)$ in the gamma distribution of shape parameter $x$ and scale parameter $1$ is, for $p\in (0,1)$, given implicitly as
\begin{equation}
\label{eq:m}
\int_0^{m_p(x)}e^{-t}t^{x-1}\, dt=p\Gamma(x).
\end{equation}
This equation can be written as $P(x,m_p(x))=p$, where $P$ is a normalized incomplete gamma function. See \cite[8.2.4]{dlmf}.
The asymptotic behaviour at infinity of the quantiles in the gamma distribution was investigated in the papers \cite{T} and \cite{daalhuis-nemes} (and for the beta distribution in \cite{T2}). We complement this by an investigation at the origin and revisit the expansion at infinity via an elementary method, inspired by \cite{CP1}. Notice that the main focus in \cite{CP1} (and \cite{CP2}) was on monotonicity properties of the median $m=m_{1/2}$, and the so-called Chen-Rubin conjecture.

The behaviour of $m$ and $m'$ at the origin was in \cite{CP1} shown to be
$$
m(x)\sim e^{-\gamma}e^{-\frac{\log 2}{x}}, \quad m'(x)\sim \frac{\log 2}{x^2}e^{-\gamma}e^{-\frac{\log 2}{x}}\quad \text{as}\ x\to 0.
$$
Introducing the function
$u(x)=e^{\nicefrac{\log 2}{x}} m(x)$
it thus follows that $u(x)\to e^{-\gamma}$ for $x\to 0$. A plot of the function $u$ leaves the impression that also its derivative has a finite and positive limit at zero. See Figure \ref{fig:u}. Let us mention that the asymptotic behaviour of $m$ and $m'$ at the origin as displayed above is not sufficient for determining the behaviour of $u'$. In fact, these relations only yield that
$$
x^2u'(x)=x^2e^{\frac{\log 2}{x}}m'(x)-(\log 2) e^{\frac{\log 2}{x}}m(x)
$$
tends to zero as $x$ tends to zero. 

We define, for $p\in (0,1)$ and $x>0$,
\begin{equation}
\label{eq:u_p}
u_p(x)=e^{-\frac{\log p}{x}}m_p(x).
\end{equation}
Experimenting with plots of these functions (for different $p$) indicate that their behaviour at the origin do not depend on $p$. This motivates the results of the paper in which the asymptotic behaviour of $u_p$ at the origin is studied. 

In Proposition \ref{prop:limit-u-n} it is shown that $u_p\in C^{\infty}([0,\infty))$ and that any derivative of $u_p$ at the origin can be computed explicitly, and does not depend on $p$. (It follows from the implicit function theorem that $m_p$ and hence also $u_p$ are $C^{\infty}$-functions on $(0,\infty)$.) In Proposition \ref{prop:m^n} and Corallary \ref{cor:m^n} the asymptotic behaviour of $u_p$ is used to obtain the asymptotic behaviour of $m_p$ and its derivatives at the origin. In particular it is proved that for any $n\geq 0$,
$$
m_p^{(n)}(x)\sim \frac{(-\log p)^n}{x^{2n}}e^{-\gamma}e^{\frac{\log p}{x}}\quad \text{as}\ x\to 0,
$$
generalizing the two relations mentioned above. See \cite{DA} for the asymptotic behaviour at zero of the quantiles in the beta distribution as a function of one of the parameters.

In Appendix \ref{sec:m-at-infinity} some remarks on the asymptotic behaviour of $m_p$ at infinity are given. Another procedure (compared to \cite{T} and \cite{daalhuis-nemes}) is described and Maple code is given. We stress that our result, Theorem \ref{thm:mp-asymp}, is also given in \cite{daalhuis-nemes}.

\begin{figure}
 \begin{center}
 \includegraphics[width=0.5\textwidth]{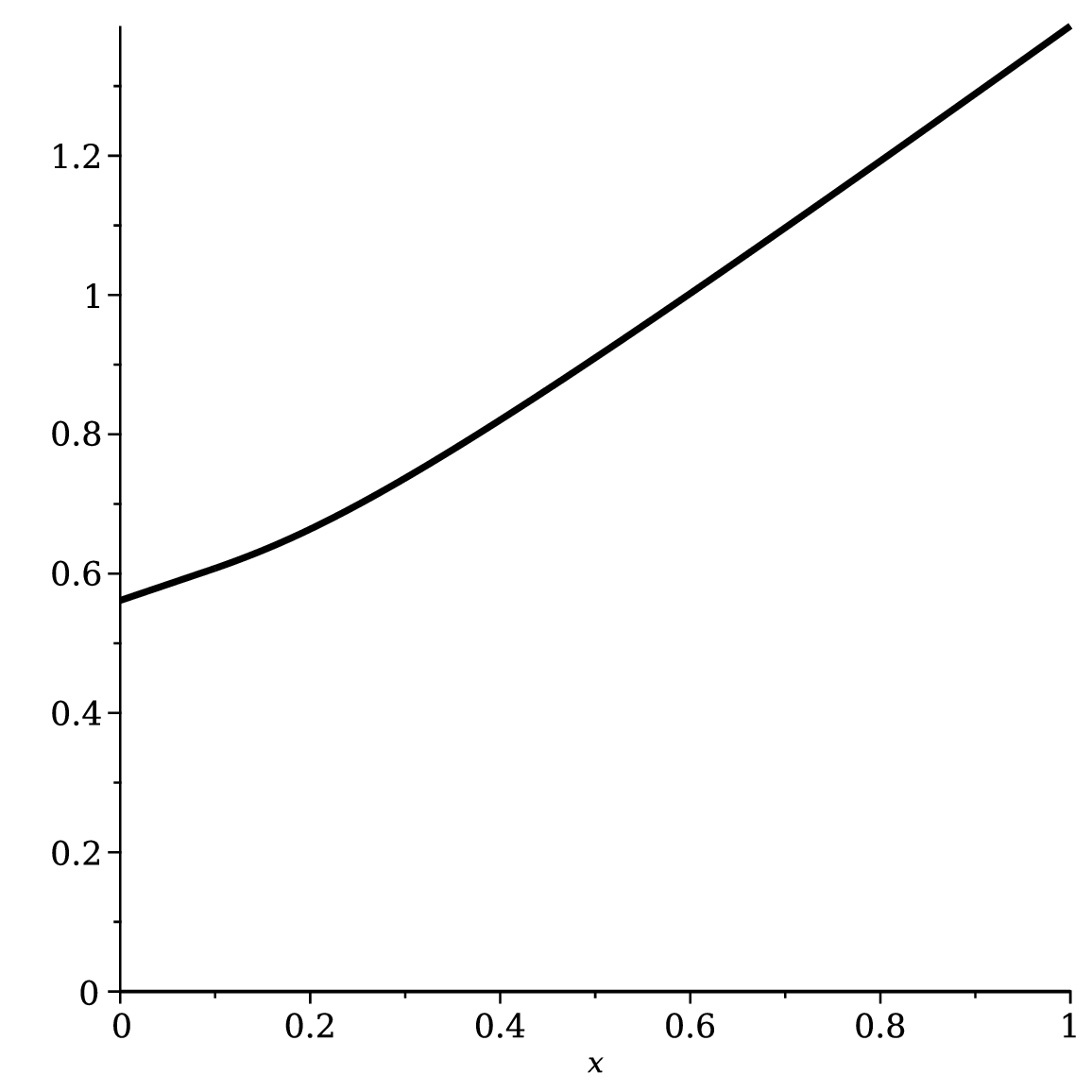}
 \caption{The graph of $u$}
 \label{fig:u}
\end{center}
\end{figure}

Let us end the introduction with a basic result about the quantiles.
\begin{lemma}
\label{lemma:1}
 The function $m_p$ is strictly increasing on $(0,\infty)$. Furthermore, $\lim_{x\to 0}m_p(x)=0$ and $\lim_{x\to \infty}m_p(x)=\infty$.
\end{lemma}
\begin{proof} The family $\{e^{-t}t^{x-1}dt/\Gamma(x)\}_{x>0}$ is a convolution semigroup of probabilities on $[0,\infty)$ and the assertions in the lemma hold for any convolution semigroup of probabilities on $[0,\infty)$. This is proved in the same way as \cite[Proposition 2.1]{CP1}, replacing $\nicefrac{1}{2}$ by $p$.
\end{proof}

\section{The behaviour of $u_p$ at the origin}
The limit of $u_p$ (defined in \eqref{eq:u_p}) at zero is given in Proposition \ref{prop:up-at-zero} and depends on the next lemma.
\begin{lemma}
\label{lemma:m_p^x}
 We have $\lim_{x\to 0}m_p(x)^x=p$.
 \end{lemma}
\begin{proof}
 Integration by parts yields
 \begin{align}
 m_p(x)^xe^{-m_p(x)}&=x\int_0^{m_p(x)}t^{x-1}e^{-t}\, dt-\int_0^{m_p(x)}t^xe^{-t}\, dt\nonumber \\
 &=p\Gamma(x+1)-\int_0^{m_p(x)}t^{x}e^{-t}\, dt, \label{eq:m_p^x}
 \end{align}
and since $m_p(x)\to 0$ for $x\to 0$ the proof is complete.
\end{proof}

\begin{prop}
\label{prop:up-at-zero}
 We have $\lim_{x\to 0}u_p(x)=e^{-\gamma}$.
\end{prop}
\begin{proof}
 Differentiation of \eqref{eq:m_p^x} yields after some manipulation
 $$
 m_p(x)^x\ell_p'(x)e^{-m_p(x)}=p\Gamma'(x+1)-\int_0^{m_p(x)}t^xe^{-t}\log t\, dt,
 $$
where $\ell_p(x)=x\log m_p(x)$. From Lemma \ref{lemma:1} and Lemma \ref{lemma:m_p^x} it follows that $\lim_{x\to 0}\ell_p'(x)=\Gamma'(1)=-\gamma$. Applying l'Hospital's rule we obtain
$$
\lim_{x\to 0}\log u_p(x)=\lim_{x\to 0}\frac{\ell_p(x)-\log p}{x}=\lim_{x\to 0}\ell_p'(x)=-\gamma
$$
and this completes the proof.\end{proof}

The function $u_p$ appears when making the change of variable $s=e^{-\lx}t$ in the relation \eqref{eq:m} which then becomes
$$
\int_0^{u_p(x)}e^{-se^{\lx}+(x-1)\log s}\, ds=\Gamma(x).
$$
Differentiation gives us
\begin{align*}
u_p'(x)e^{-u_p(x)e^{\lx}}u_p(x)^{x-1}&=-\int_0^{u_p(x)}e^{-se^{\lx}}s^{x}\, ds \, e^{\lx}\tfrac{\log p}{x^2}\\
&\phantom{=}-\int_0^{u_p(x)}e^{-se^{\lx}}s^{x-1}\log s \, ds+\Gamma'(x).
\end{align*}
The limit of the sum of the two last terms on the right-hand side evaluates as $\infty-\infty$. In order to overcome this difficulty we rewrite as follows.
\begin{align*}
u_p'(x)e^{-u_p(x)e^{\lx}}u_p(x)^{x-1}&=-\int_0^{u_p(x)}e^{-se^{\lx}}s^{x}\, ds\, e^{\lx}\tfrac{\log p}{x^2}\\
&\phantom{=}+\int_0^{u_p(x)}\left(1-e^{-se^{\lx}}\right)s^{x-1}\log s \, ds\\
&\phantom{=}+\int_0^{u_p(x)}\frac{e^{-s}-1}{s}s^{x}\log s \, ds+
\int_{u_p(x)}^{\infty}\frac{e^{-s}}{s}s^{x}\log s \, ds.
\end{align*}
Putting
\begin{align*}
I_1(x)&=\int_0^{u_p(x)}e^{-se^{\lx}}s^{x}\, ds,\\
I_2(x)&=\int_0^{u_p(x)}\left(1-e^{-se^{\lx}}\right)s^{x-1}\log s \, ds,\\
I_3(x)&=\int_0^{u_p(x)}\frac{e^{-s}-1}{s}s^{x}\log s \, ds,\\
I_4(x)&=\int_{u_p(x)}^{\infty}\frac{e^{-s}}{s}s^{x}\log s \, ds,
\end{align*}
the relation above can be written as
\begin{equation}
 \label{eq:diff}
u_p'(x)e^{-m_p(x)}u_p(x)^{x-1}=-I_1(x)e^{\lx}\tfrac{\log p}{x^2}+I_2(x)+I_3(x)+I_4(x).
\end{equation}
\begin{prop}
\label{prop:updiff}
 The limit $\lim_{x\to 0}u_p'(x)$ exists and equals
 $$
 e^{-\gamma}\left(\int_0^{e^{-\gamma}}\frac{e^{-s}-1}{s}\log s \, ds+\int_{e^{-\gamma}}^{\infty}\frac{e^{-s}}{s}\log s \, ds\right).
 $$
\end{prop}
\begin{proof}
It is easy to see that (writing $u_p(0)$ instead of $e^{-\gamma}$)
$$
\lim_{x\to 0}I_1(x)=u_p(0),
$$
$$
\lim_{x\to 0}I_3(x)=\int_0^{u_p(0)}\frac{e^{-s}-1}{s}\log s \, ds,
$$
$$
\lim_{x\to 0}I_4(x)=\int_{u_p(0)}^{\infty}\frac{e^{-s}}{s}\log s \, ds,
$$
and using the estimate $|1-e^{-y}|\leq y$ for $y\geq 0$ we obtain
$$
\lim_{x\to 0}I_2(x)=0.
$$
These relations yield
$$
\lim_{x\to 0}u_p'(x)e^{-m_p(x)}u_p(x)^{x-1}=\int_0^{u_p(0)}\frac{e^{-s}-1}{s}\log s \, ds+\int_{u_p(0)}^{\infty}\frac{e^{-s}}{s}\log s \, ds,
$$
i.e.\
$$
\lim_{x\to 0}u_p'(x)=u_p(0)\left(\int_0^{u_p(0)}\frac{e^{-s}-1}{s}\log s \, ds+\int_{u_p(0)}^{\infty}\frac{e^{-s}}{s}\log s \, ds\right).
$$
This completes the proof.
\end{proof}

By differentiation of the relation \eqref{eq:diff} one obtains an involved expression containing only one term with $u_p''(x)$. It turns out that it is possible to compute the limit $u_p''(0)$ and express it in terms of $u_p'(0)$ and $u_p(0)$. The next step is to show that any derivative $u_p^{(n)}(x)$ has a limit as $x$ tends to $0$, and hence that $u_p\in C^{\infty}([0,\infty))$.
\begin{prop}
\label{prop:limit-u-n}
The limit
$$
\lim_{x\to 0}u_p^{(n)}(x)
$$
exists for any $n\geq 0$ and is independent of $p$.
\end{prop}

The proof of this proposition is carried out by an inductive argument. In order to do so some computational lemmas are needed. We shall also use the notation $\partial_x^j$ to denote the $j$th derivative when it is convenient.
 As one would expect, the function 
 \begin{equation}
 \label{eq:alpha_p}
 \alpha_p(x)=e^{\lx}
 \end{equation}
 plays a key role in the development. Its behaviour is determined by the following lemma.
\begin{lemma}
\label{lemma:p_n}
For any $n\geq 1$,
 $$
 \partial_x^n\left(e^{-1/x}\right)=p_n(1/x)e^{-1/x},
 $$
 where $p_n(t)=\sum_{k=n+1}^{2n}a_{k,n}t^k$ with $a_{k,n}\in \mathbb Z$. Furthermore $a_{2n,n}=1$, $a_{2n-1,n}=-n(n-1)$ and $a_{2n-2,n}=n(n-1)^2(n-2)/2$.
\end{lemma}
\begin{proof}
 This is a standard induction argument, and the sequence $\{p_n\}$ is easily shown to satisfy $p_{n+1}(t)=t^2(p_n(t)-p_n'(t))$.
\end{proof}
\begin{lemma}
\label{lemma:m-k}
Suppose that the limit $\lim_{x\to 0}u_p^{(j)}(x)$ exists in $\mathbb R$ for $j=0,\ldots,q$. Then $\lim_{x\to 0}m_p^{(j)}(x)=0$ for $j=0,\ldots,q$.
\end{lemma}

\begin{proof}This follows from the relation (with $\alpha_p$ as in \eqref{eq:alpha_p})
$$
m_p^{(j)}(x)=\sum_{l=0}^j\binom{j}{l}\alpha_p^{(l)}(x)u_p^{(j-l)}(x),
$$
combined with the fact that $\alpha_p^{(l)}(x)\to 0$ for $x\to 0$ for all $l\geq 0$.
\end{proof}

\begin{lemma}
 \label{lemma:I_1}
 $$
 I_1(x)=u_p(x)^{x+1}\left(\frac{1}{x+1}+\sum_{k=1}^{\infty}\frac{(-1)^{k}}{k!}\frac{m_p(x)^{k}}{x+k+1}\right).
 $$
\end{lemma}
\begin{proof}
 Using the Taylor series expansion of the exponential function we get
 \begin{align*}
 I_1(x)&=\sum_{k=0}^{\infty}\frac{(-1)^k}{k!}e^{\lx k}\int_0^{u_p(x)}s^{x+k}\, ds\\
 &=\sum_{k=0}^{\infty}\frac{(-1)^k}{k!}e^{\lx k}\frac{u_p(x)^{x+k+1}}{x+k+1},
 \end{align*}
 which gives the desired formula.
\end{proof}
\begin{lemma}
\label{lemma:I_2}
 \begin{align*}
 I_2(x)&=u_p(x)^x\log u_p(x)\sum_{k=0}^{\infty}\frac{(-1)^k}{(k+1)!}\frac{m_p(x)^{k+1}}{x+k+1}\\
 &\phantom{M}-u_p(x)^x\sum_{k=0}^{\infty}\frac{(-1)^k}{(k+1)!}\frac{m_p(x)^{k+1}}{(x+k+1)^2}.
 \end{align*}
\end{lemma}
\begin{proof}
 This follows the same lines as the proof of Lemma \ref{lemma:I_1}, noticing that
 $$
 \int_0^{u_p(x)}s^{x+k}\log s\,ds=\frac{u_p(x)^{x+k+1}}{x+k+1}\log u_p(x)-\frac{u_p(x)^{x+k+1}}{(x+k+1)^2}.
 $$
\end{proof}

\begin{lemma}
\label{lemma:limit-I-1-I-2}
 Suppose that $\lim_{x\to 0}u_p^{(j)}(x)$ exists in $\mathbb R$ for $j=0,\ldots,q$. Then
 \begin{align*}
 \lim_{x\to 0}\partial_x^jI_1(x)&=\lim_{x\to 0}\partial_x^j\left(\frac{u_p(x)^{x+1}}{x+1}\right),\ \text{and}\\
 \lim_{x\to 0}\partial_x^jI_2(x)&=0
\end{align*}
for $j=0,\ldots,q$.
\end{lemma}
\begin{proof}
 This follows from Lemma \ref{lemma:I_1} and Lemma \ref{lemma:I_2}, noticing that termwise differentiation is permitted and that the limit of any term involving $m_p(x)$ in a positive power equals zero.
\end{proof}

The integrals $I_3$ and $I_4$ are analyzed using Lemma \ref{lemma:integrals}. Here, $u_p$ could be replaced by any positive and smooth function, and the lemma is proved by induction.
\begin{lemma}
\label{lemma:integrals}
 The relation
 \begin{align*}
 \partial_x^j\int_0^{u_p(x)}g(s)\log s\, s^x\, ds&=\sum_{l=1}^j\partial_x^{j-l}\left(u_p'(x)u_p(x)^xg(u_p(x))(\log u_p(x))^l\right)\\
 &\phantom{M}+\int_0^{u_p(x)}g(s)(\log s)^{j+1}\, s^x \, ds
 \end{align*}
 holds when $g\in C^{\infty}([0,\infty))$ and the relation
 \begin{align*}
 \partial_x^j\int_{u_p(x)}^{\infty}h(s)\log s\, s^x\, ds&=-\sum_{l=1}^j\partial_x^{j-l}\left(u_p'(x)u_p(x)^xh(u_p(x))(\log u_p(x))^l\right)\\
 &\phantom{M}+\int_{u_p(x)}^{\infty}h(s)(\log s)^{j+1}\, s^x \, ds
 \end{align*}
 holds when $h\in C^{\infty}((0,\infty))$ and decays exponentially at infinity.
\end{lemma}

{\it Proof of Proposition \ref{prop:limit-u-n}.}
 Suppose that for $j=0,\ldots,q$, $\lim_{x\to 0}\partial_x^ju_p(x)$ exists in $\mathbb R$ and is independent of $p$.  We set out to verify that $\lim_{x\to 0}\partial_x^{q+1}u_p(x)$ exists in $\mathbb R$ and is independent of $p$.
Differentiation of \eqref{eq:diff} $q$ times yields 
\begin{equation}
 \label{eq:diff-p-times}
\partial_x^q\left(u_p'(x)e^{-m_p(x)}u_p(x)^{x-1}\right)=\partial_x^q\left(I_1(x)\alpha_p'(x)+I_2(x)+I_3(x)+I_4(x)\right).
\end{equation}
The left hand side becomes
\begin{align*}
 \lefteqn{\partial_x^q\left(u_p'(x)e^{-m_p(x)}u_p(x)^{x-1}\right)=}\\
 &u_p^{(q+1)}(x)e^{-m_p(x)}u_p(x)^{x-1}
 +\sum_{j=0}^{q-1}\binom{q}{j}u_p^{(j+1)}(x)\partial_x^{q-j}\left(e^{-m_p(x)}u_p(x)^{x-1}\right).
\end{align*}
Notice that according to Lemma \ref{lemma:m-k}, $m_p^{(k)}(x)=o(1)$ as $x\to 0$ for $k\leq q$, and therefore we get
$$
\partial_x^{q-j}\left(e^{-m_p(x)}u_p(x)^{x-1}\right)=\partial_x^{q-j}u_p(x)^{x-1}+o(1).
$$
This gives
\begin{align*}
\partial_x^q\left(u_p'(x)e^{-m_p(x)}u_p(x)^{x-1}\right)&=u_p^{(q+1)}(x)e^{-m_p(x)}u_p(x)^{x-1}+\\
&\phantom{=}\sum_{j=0}^{q-1}\binom{q}{j}u_p^{(j+1)}(x)\partial_x^{q-j}u_p(x)^{x-1}+o(1).
\end{align*}
The right hand side of \eqref{eq:diff-p-times} is rewritten as
\begin{align*}
\partial_x^{q}&\left(I_1(x)\alpha_p'(x)+I_2(x)+I_3(x)+I_4(x)\right)\\
&=\sum_{j=0}^{q}\binom{q}{j}\partial_x^{j}I_1(x)\alpha_p^{(1+q-j)}(x)+\partial_x^{q}I_2(x)+\partial_x^{q}I_3(x)+ \partial_x^{q}I_4(x).
\end{align*}
From Lemma \ref{lemma:p_n} it follows that $\lim_{x\to 0}\alpha_p^{(k)}(x)=0$ for any $k\geq 0$ and by Lemma \ref{lemma:limit-I-1-I-2} and Lemma \ref{lemma:integrals} we see that, with $g(s)=(e^{-s}-1)/s$ and $h(s)=e^{-s}/s$,
\begin{align*}
\partial_x^{q}&\left(I_1(x)\alpha_p'(x)+I_2(x)+I_3(x)+I_4(x)\right)\\
&=\sum_{j=1}^q\partial_x^{q-j}\left(u_p'(x)u_p(x)^x(g(u_p(x))-h(u_p(x)))(\log u_p(x))^j\right)
\\
&\phantom{=}+\int_0^{u_p(x)}g(s)(\log s)^{q+1}\, s^x \, ds
 +\int_{u_p(x)}^{\infty}h(s)(\log s)^{q+1}\, s^x \, ds+o(1)\\
 &=-\sum_{j=1}^q\partial_x^{q-j}\left(u_p'(x)u_p(x)^{x-1}(\log u_p(x))^j\right)\\
 &\phantom{=}+\int_0^{u_p(0)}g(s)(\log s)^{q+1}\, ds
 +\int_{u_p(0)}^{\infty}h(s)(\log s)^{q+1}\, ds+o(1).
 \end{align*}
 Combination of these relations yield
 \begin{align*}
  \lefteqn{u_p^{(q+1)}(x)e^{-m_p(x)}u_p(x)^{x-1}}\\
  &=
  -\sum_{j=1}^q\partial_x^{q-j}\left(u_p'(x)u_p(x)^{x-1}(\log u_p(x))^j\right)
  -\sum_{j=0}^{q-1}\binom{q}{j}u_p^{(j+1)}(x)\partial_x^{q-j}u_p(x)^{x-1}\\
 &\phantom{=}+\int_0^{u_p(0)}g(s)(\log s)^{q+1}\, ds
 +\int_{u_p(0)}^{\infty}h(s)(\log s)^{q+1}\, ds+o(1).
 \end{align*}
The limit, as $x\to 0$, of the right hand side of this relation is an expression involving  $u_p^{(j)}(0)$, $j\leq q$, which by our induction hypothesis are independent of $p$.
\hfill $\square$

Letting 
\begin{equation*}
s_q=\int_0^{e^{-\gamma}}g(s)(\log s)^{q+1}\, ds
 +\int_{e^{-\gamma}}^{\infty}h(s)(\log s)^{q+1}\, ds,
 \end{equation*}
 where $g(s)=(e^{-s}-1)/s$ and $h(s)=e^{-s}/s$ it is straight forward to verify that
 \begin{align}
 \label{eq:s1}u_p'(0)&=e^{-\gamma}s_1,\\
 \label{eq:s2}u_p''(0)&=e^{-\gamma}\left(s_2+s_1^2+2\gamma s_1\right).
 \end{align}
\section{Asymptotic behaviour of the quantiles at zero}
\label{sec:m-at-zero}
The first result is simply the Taylor expansion of $u_p$.
\begin{prop}
 As $x$ tends to $0$ we have, for any $n$,
 $$
 m_p(x)=e^{\lx}\left(\sum_{k=0}^n\frac{u_p^{(k)}(0)}{k!}x^k+O\left(x^{n+1}\right) \right).
 $$
\end{prop}
\begin{prop}
\label{prop:m^n}
If $\{p_k\}$ are the polynomials from Lemma \ref{lemma:p_n} then the relation
 $$
 m_p^{(n)}(x)=\sum_{k=0}^n\binom{n}{k}(-\log p)^{-k}p_k\left(-\lx\right)e^{\lx}u_p^{(n-k)}(x)
 $$
 holds.
\end{prop}
\begin{proof}
 From Lemma \ref{lemma:p_n} we get that
 $$
 \partial_x^k\left(e^{\lx}\right)=(-\log p)^{-k}p_k\left(-\lx\right)e^{\lx}
 $$
 and the result follows from Leibniz' rule.
\end{proof}
Proposition \ref{prop:m^n} gives us information on the asymptotic expansion of the derivatives of $m_p$:
\begin{cor} 
\label{cor:m^n}
For any $n\geq 0$ we have, recalling that $a_{l,k}$ are the coefficients of $p_k$ in Lemma \ref{lemma:p_n},
\begin{align*}
e^{-\lx}m_p^{(n)}(x)
&=\sum_{k=0}^n\sum_{l=1}^{k}\sum_{j=0}^{2k}\binom{n}{k}a_{l+k,k}\left(-\log p\right)^{l}\frac{u_p^{(n-k+j)}(0)}{j!\, x^{k+l-j}}+O(x)\\
&=\sum_{k=0}^{2n}\frac{z_{k,n}}{x^k}+O(x)
\end{align*}
for $x\to 0$,
where
\begin{align*}
z_{2n,n}&=e^{-\gamma}(-\log p)^{n},\ \text{and}\\
z_{2n-1,n}&=e^{-\gamma}(-\log p)^{n-1}\left(-s_1\log p-n(n-1)\right).
\end{align*}
\end{cor}
\begin{proof}
The first identity follows from the expansion
$$
u_p^{(n-k)}(x)=\sum_{j=0}^{2k}\frac{u_p^{(n-k+j)}(0)}{j!}x^j+O(x^{2k+1})
$$
and Proposition \ref{prop:m^n}.
To find $z_{2n,n}$ we identify the set of triples $(k,l,j)$ of indices for which $k+l-j=2n$. This is $\{(n,n,0)\}$.  To find $z_{2n-1,n}$ the corresponding set is $\{(n,n,1),(n,n-1,0)\}$. 
The formulae follow by computation, using \eqref{eq:s1}.
\end{proof}
\begin{rem}
We have
\begin{align*}
z_{2n-2,n}&=e^{-\gamma}(-\log p)^{n-2}\left(\frac{(\log p)^2}{2}(s_2+s_1^2+2\gamma s_1)+n(n-1)s_1\log p\right.\\
&\phantom{=}\left. +\frac{n(n-1)^2(n-2)}{2}-\frac{n^2(n-1)^2(n-2)}{2}s_1\log p\right).
\end{align*}
To see this notice that the set of indices is $\{(n,n,2),(n,n-1,1),(n,n-2,0),(n-1,n-1,0)\}$, use \eqref{eq:s2} and Lemma \ref{lemma:p_n}.
\end{rem}
\appendix

\section{Asymptotic behaviour at infinity; Maple code}
\label{sec:m-at-infinity}
The purpose of this appendix is to outline a method where the asymptotic behaviour of the function
$$
\varphi_p(x)=\log\left(\frac{x}{m_p(x)}\right)
$$
is used to find the asymptotic behaviour of $m_p$. Our method is quite elementary, basically only Lebesgue's dominated convergence theorem is used. The notation
$$
f(x)\sim \sum_{k=0}^{\infty}\frac{\beta_k}{x^k}, \quad x\to \infty
$$
means that for any $N$, $f(x)-\sum_{k=0}^N\beta_k/x^k=O(x^{-N-1})$ as $x\to \infty$.
Inserting Euler's integral representation of $\Gamma(x)$ in \eqref{eq:m}, changing variables to $u=\log (x/t)$ and performing some manipulations we obtain
\begin{align}
\label{eq:star}
\int_0^{\varphi_p(x)}e^{-x(e^{-u}+u)}\, du&=
(1-p)\int_0^{\infty}e^{-x(e^{-u}+u)}\, du-p\int_0^{\infty}e^{-x(e^{u}-u)}\, du.
\end{align}
By multiplication by $e^x$ and  the change of variable $s=\sqrt{x}u$ we get
\begin{align}
\label{eq:sqrt-varphi_p}
\int_0^{\sqrt{x}\varphi_p(x)}e^{-s^2h(-\nicefrac{s}{\sqrt{x}})}\, ds
&=(1-p)J_1(x)-pJ_2(x),
\end{align}
where $h$ is the positive and increasing function  on $\mathbb R$ defined as
\begin{equation}
\label{eq:def-h}
 h(u)=\frac{e^{u}-1-u}{u^2}=\sum_{k=0}^{\infty}\frac{1}{(k+2)!}u^k,
\end{equation}
and
$$
J_1(x)=\int_0^{\infty}e^{-s^2h(-\nicefrac{s}{\sqrt{x}})}\, ds,\quad 
J_2(x)=\int_0^{\infty}e^{-s^2h(\nicefrac{s}{\sqrt{x}})}\, ds.
$$ 
The asymptotic behaviour of $\varphi_p$ is found using the behaviour of the right hand side of \eqref{eq:sqrt-varphi_p} (see Proposition \ref{prop:XXX}) and by relating the behaviour of $\varphi_p$ to the integral in the left hand side of \eqref{eq:sqrt-varphi_p}. 
A fundamental sequence of polynomials is introduced in the next lemma (which is proved by induction).
\begin{lemma}
 \label{lemma:diff}
 Let $n\geq 0$. There exists a polynomial $Q_n$ of $n+1$ variables such that
 $$
 \partial_v^ne^{-s^2h(v)}=Q_n(s,h'(v),\ldots,h^{(n+1)}(v))e^{-s^2h(v)}.
 $$ 
 For fixed $v\in \mathbb R$, $Q_n$ is an even polynomial of $s$ of degree $2n$.
\end{lemma}
The first few of these polynomials are $Q_1(s,y_1)=-s^2y_1$,
$Q_2(s,y_1,y_2)=s^4y_1^2-s^2y_2$ and
$Q_3(s,y_1,y_2,y_3)=-s^6y_1^3+3s^4y_1y_2-s^2y_3$.
\begin{lemma}
\label{lemma:J(x)}
The function 
$$
J(x)=\int_0^{\infty}e^{-s^2h(xs)}\, ds
$$
belongs to $C^{\infty}(\mathbb R)$ and
$$
J^{(n)}(x)=\int_0^{\infty}s^nQ_n(s,h'(xs),\ldots,h^{(n+1)}(xs))e^{-s^2h(xs)}\, ds,
$$
$Q_n$ being the polynomial introduced in Lemma \ref{lemma:diff}.
\end{lemma}
\begin{proof}
By the chain rule,
$$
\partial_x^ne^{-s^2h(xs)}=s^nQ_n(s,h'(xs),\ldots,h^{(n+1)}(xs))e^{-s^2h(xs)}.
$$
The proof of the lemma consists now of obtaining an integrable majorant of this expression for $-R\leq x\leq R$ where $R>0$ is arbitrary. Assume first that $x\geq 0$. For $u\geq 0$ we have an exponential upper bound on the derivatives of $h$:
$$
h^{(j)}(u)=\partial_u^j\left(\sum_{k=0}^{\infty}\frac{u^k}{(k+2)!}\right)=\sum_{k=0}^{\infty}\frac{1}{(k+j+2)(k+j+1)}\frac{u^k}{k!}\leq e^u.
$$
Writing
$$
s^nQ_n(s,y_1,\ldots,y_{n+1})=\sum_{(m_1,\ldots,m_{n+1})}\alpha_{(m_1,\ldots,m_{n+1})}s^{m_1}y_1^{m_2}\cdots y_{n+1}^{m_{n+1}}
$$
we get
$$
\left|\partial_x^ne^{-s^2h(xs)}\right|\leq
\sum_{(m_1,\ldots,m_{n+1})}|\alpha_{(m_1,\ldots,m_{n+1})}|s^{m_1}e^{m_2Rs}\cdots e^{m_{n+1}Rs}e^{-s^2/2},
$$
which is integrable on $[0,\infty)$.
Assuming that $x\leq 0$ we remark that any derivative of $h$ is bounded on $(-\infty,0]$ so that, for some polynomial $r(s)$,
$$
|\partial_x^ne^{-s^2h(xs)}|\leq r(s)e^{-s^2h(xs)}\leq r(s)e^{-s^2h(-Rs)},
$$
which is integrable on $[0,\infty)$. 
\end{proof}
To ease notation, let
\begin{equation}
\label{eq:defQ}
Q_k(s)=Q_k(s,h'(0),\ldots,h^{(k+1)}(0)).
\end{equation}
\begin{prop}
\label{prop:XXX}
As $x\to\infty$ the functions $J_1$ and $J_2$ admit the expansions
\begin{align*}
J_1(x)&\sim\sum_{k=0}^{\infty} \frac{(-1)^k}{k!}\frac{c_k}{x^{k/2}}\quad \text{and}\quad 
J_2(x)\sim\sum_{k=0}^{\infty} \frac{1}{k!}\frac{c_k}{x^{k/2}},
\end{align*}
where
$$
c_k=\int_0^{\infty}s^kQ_k(s)e^{-s^2/2}\, ds.
$$    
\end{prop}
\begin{proof}
Immediate using $J_1(x)=J(-\nicefrac{1}{\sqrt{x}})$ and $J_2(x)=J(\nicefrac{1}{\sqrt{x}})$.  
\end{proof}
The values of $c_k$ can be computed via the explicit formulas $Q_k$, the derivatives of $h$ at the origin and the elementary formula
$$
\int_0^{\infty}s^je^{-s^2/2}\, ds =2^{(j-1)/2}\Gamma\left( (j+1)/2\right).
$$
We remark that 
$J_2(x)=x^{\nicefrac{3}{2}-x}e^x\Gamma(x,x)$, where $\Gamma(x,x)$ is the complementary incomplete gamma function (see \cite[8.2.2]{dlmf}) and this can also be used to obtain the  expansion of $J_2$. For sharp error bounds and exponential asymptotics, see \cite{nemes}. 

Notice that when $L\in \mathbb R$, the relation
\begin{equation}
\label{eq:d-expansion}
\int_0^Le^{-s^2h(-\nicefrac{s}{\sqrt{x}})}\, ds=
\sum_{k=0}^{\infty} \frac{(-1)^k}{k!}\frac{d_k}{x^{k/2}},\quad x\in \mathbb R,
\end{equation}
where
$$
d_k=\int_0^{L}s^kQ_k(s)e^{-s^2/2}\, ds.
$$ 
is obtained directly from Lemma \ref{lemma:diff} and Taylor expansion.

 We let $L_p$ denote the $(1-p)$-quantile in the standard Gaussian distribution:
 \begin{equation}
\label{eq:L_p}
\int_0^{L_p}e^{-s^2/2}\, ds=(1-2p)\int_0^{\infty}e^{-s^2/2}\, ds=(1-2p)\sqrt{\frac{\pi}{2}}.
\end{equation}
\begin{lemma}
\label{lemma:asymp-varphi-ver-2}
The function $\varphi_p$ admits an asymptotic expansion at infinity:
$$
\varphi_p(x)\sim \sum_{k=1}^{\infty}\frac{b_k}{x^{k/2}}=\frac{L_p}{x^{\nicefrac{1}{2}}}+\frac{\nicefrac{1}{3}+\nicefrac{L_p^2}{6}}{x}+\frac{\nicefrac{L_p^3}{36}+\nicefrac{5L_p}{36}}{x^{\nicefrac{3}{2}}}+\ldots\ .
$$
\end{lemma}

\begin{proof}
From \eqref{eq:sqrt-varphi_p} and Proposition  \ref{prop:XXX} it follows that
$$
\lim_{x\to \infty}\int_0^{\sqrt{x}\varphi_p(x)}
e^{-s^2h(-\nicefrac{s}{\sqrt{x}})}\, ds= (1-2p)\sqrt{\frac{\pi}{2}}.
$$
For $s\geq 0$ and $x\geq 1$, $h(-\nicefrac{s}{\sqrt{x}})\geq h(-s)$ so that
$e^{-s^2h(-\nicefrac{s}{\sqrt{x}})}\leq e^{-s^2h(-s)}=e^{-(e^{-s}+s-1)}$,
which is integrable. For $s\leq 0$, $h(-\nicefrac{s}{\sqrt{x}})\geq h(0)=\nicefrac{1}{2}$ so that
$e^{-s^2h(-\nicefrac{s}{\sqrt{x}})}\leq e^{-\nicefrac{s^2}{2}}$,
which is also integrable. Let $\alpha=\limsup_{x\to \infty}\sqrt{x}\varphi_p(x)$ and choose a sequence $\{x_n\}$ tending to infinity such that $\lim_{n\to \infty}\sqrt{x_n}\varphi_p(x_n)=\alpha$. By dominated convergence it follows that
$$
\int_0^{\alpha}
e^{-\nicefrac{s^2}{2}}\, ds= (1-2p)\sqrt{\frac{\pi}{2}}.
$$
Repeating the argument using $\beta=\liminf_{x\to \infty}\sqrt{x}\varphi_p(x)$ we see that $\alpha=\beta=L_p$, where $L_p$ is given in \eqref{eq:L_p} and hence that $\lim_{x\to \infty}\sqrt{x}\varphi_p(x)=L_p$.

The next step is to rewrite relation \eqref{eq:sqrt-varphi_p} in the form
\begin{align}
\label{eq:sqrt-varphi_p-again-again}
\int_{L_p}^{\sqrt{x}\varphi_p(x)}
e^{-s^2h(-\nicefrac{s}{\sqrt{x}})}\, ds&= (1-p)J_1(x)
-pJ_2(x)-\int_{0}^{L_p}
e^{-s^2h(-\nicefrac{s}{\sqrt{x}})}\, ds.
\end{align}
By \eqref{eq:d-expansion} (with $L=L_p$) and Proposition \ref{prop:XXX} the right hand side of \eqref{eq:sqrt-varphi_p-again-again} admits the asymptotic expansion  
$$
\sum_{k=0}^{\infty}\frac{A_k}{x^{k/2}},
$$
where $A_k=\left((1-p)(-1)^kc_{k}-pc_k-(-1)^kd_k\right)/k!$. (Notice that $A_0=0$.) 
Looking at the left-hand side of \eqref{eq:sqrt-varphi_p-again-again}, we have
\begin{equation}
    \label{eq:lhs}
\int_{L_p}^{\sqrt{x}\varphi_p(x)}
e^{-s^2h(-\nicefrac{s}{\sqrt{x}})}\, ds= \sum_{k=0}^{\infty}\frac{(-1)^k}{k!}\int_{L_p}^{\sqrt{x}\varphi_p(x)}s^kQ_k(s)e^{-s^2/2}\, ds\, \frac{1}{x^{k/2}}.
\end{equation}
The Taylor series of $q_k(s)=s^kQ_k(s)e^{-s^2/2}$ centered at $L_p$
is inserted in \eqref{eq:lhs} and we thus get
\begin{equation}
\label{eq:asymp-idea}
\sum_{k,j=0}^{\infty}\frac{(-1)^kq_k^{(j)}(L_p)}{k!(j+1)!}\left(\sqrt{x}\varphi_p(x)-L_p\right)^{j+1} \frac{1}{x^{k/2}}=\sum_{k=1}^{N}\frac{A_k}{x^{k/2}}+o\left(\frac{1}{x^{N/2}}\right).
\end{equation}
We prove, using \eqref{eq:asymp-idea}, that there is a sequence $\{a_k\}$ such that, for any $N\geq 0$, 
\begin{equation}
\label{eq:induc}
\sqrt{x}\varphi_p(x)-L_p=\sum_{k=1}^N\frac{a_k}{x^{k/2}}+o\left(\frac{1}{x^{N/2}}\right).
\end{equation}
(This proves the lemma.) For $N=0$ it is proved above. For $N=1$, \eqref{eq:asymp-idea} yields
$$
\sqrt{x}\left(\sqrt{x}\varphi_p(x)-L_p\right)\sum_{j=0}^{\infty}\frac{q_0^{(j)}(L_p)}{(j+1)!}\left(\sqrt{x}\varphi_p(x)-L_p\right)^{j} =A_1+o\left(1\right).
$$
Since the infinite sum on the left-hand side equals $q_0(L_p)+o(1)$ it follows that
$$
a_1=\nicefrac{A_1}{q_0(L_p)}=\nicefrac{1}{3}\left(1+\nicefrac{L_p^2}{2}\right).
$$
Assume now that \eqref{eq:induc}  holds for $N\geq 1$.
Then first of all, for $j\geq 0$, we have, for some numbers $a_{j,l}$, 
\begin{equation}
\label{eq:jthpower}
\left(\sqrt{x}\varphi_p(x)-L_p\right)^{j+1}=\sum_{l=1+j}^{N+j}\frac{a_{j,l}}{x^{l/2}}+o\left(\frac{1}{x^{(N+j)/2}}\right).
\end{equation}
It follows from \eqref{eq:asymp-idea} that
\begin{align*}
q_0(L_p)\left(\sqrt{x}\varphi_p(x)-L_p\right)&=
-\sum_{1\leq k+j\leq N}\frac{(-1)^kq_k^{(j)}(L_p)}{k!(j+1)!}\left(\sqrt{x}\varphi_p(x)-L_p\right)^{j+1} \frac{1}{x^{k/2}}\\
&\phantom{=}+\sum_{k=1}^{N+1}\frac{A_k}{x^{k/2}}+o\left(\frac{1}{x^{(N+1)/2}}\right),
\end{align*}
since $\left(\sqrt{x}\varphi_p(x)-L_p\right)^{j+1}/x^{k/2}=o(1/x^{(N+1)/2})$ for $j+k\geq N+1$ and this shows that $\sqrt{x}\varphi_p(x)-L_p$ has an asymptotic expansion up to order $(N+1)/2$. It also shows how to compute the coefficients $a_k(=a_{0,k})$ recursively. Let $T_N$ denote the triples of non-negative integers $(k,l,j)$ such that
$$
k+j\geq 1,\quad j+1\leq l\leq  N+j,\quad k+l=N+1.
$$
Then 
\begin{align}
\label{eq:]}
a_{N+1}=\frac{1}{q_0(L_p)}\left(\sum_{(k,l,j)\in T_N}\frac{(-1)^{k+1}q_k^{(j)}(L_p)}{k!(j+1)!}a_{j,l}+A_{N+1}\right).
\end{align}
Since $T_1=\{(1,1,0),(0,2,1)\}$ and $a_{1,2}=a_1^2$ we find that 
$a_2=\nicefrac{L_p^3}{36}+\nicefrac{5L_p}{36}$.
\end{proof}
\begin{rem}
    It is evident that also $J_1'$ and $J_2'$ have asymptotic expansions at infinity. From the relation 
    \begin{align*}
        (\sqrt{x}\varphi_p(x))'e^{-(\sqrt{x}\varphi_p(x))^2}h(&-\sqrt{x}\varphi_p(x)/\sqrt{x})=(1-p)J_1'(x)-pJ_2'(x)\\
        &\phantom{=}+\int_0^{\sqrt{x}\varphi_p(x)}e^{-2^2h(-s/\sqrt{x})}h'(-s/\sqrt{x})s^3\, ds\frac{1}{2x^{3/2}}
    \end{align*}
    it follows that also $(\sqrt{x}\varphi_p(x))'$ has an asymptotic expansion at infinity.
\end{rem}
\begin{thm}
\label{thm:mp-asymp}
The $p$-quantile $m_p$ has an asymptotic expansion at infinity:
$$
m_p(x)\sim \sum_{n=-2}^{\infty}\frac{\tau_n}{x^{n/2}}=x-L_p\, x^{\nicefrac{1}{2}}+\nicefrac{(L_p^2-1)}{3}+(\nicefrac{L_p(7-L_p^2)}{36})\, x^{-\nicefrac{1}{2}}+\ldots.
$$
\end{thm}
\begin{proof}
This follows from Lemma \ref{lemma:asymp-varphi-ver-2} and  
the relation $m_p(x)=xe^{-\varphi_p(x)}$.
\end{proof}
This result was obtained in \cite[Theorem 1.3]{daalhuis-nemes}, where it was also obtained that the coefficient $\tau_n$ is a polynomial of degree $n+2$ with rational coefficients independent of $p$, evaluated at $L_p$. It is an even (resp.\ odd) polynomial when $n$ is even (resp.\ odd). These facts, except the degree being $n+2$, can also be obtained via the method outlined above but we skip the details.

Maple code for computing the polynomials in the expansions of $\varphi_p$ and $m_p$ is given below. We get (in accordance with \cite[Table 2]{daalhuis-nemes} since $L_{1-p}=-L_p$)
\begin{align*}
   \tau_2(t)&=-\nicefrac{t^4}{270}-\nicefrac{t^2}{810}+\nicefrac{8}{405}\\
   \tau_3(t)&=-\nicefrac{t^5}{4320}-\nicefrac{8t^3}{1215}+\nicefrac{433t}{38880}\\
   \tau_4(t)&=\nicefrac{t^6}{17010}-\nicefrac{t^4}{840}-\nicefrac{923t^2}{204120}+\nicefrac{184}{25515}\\
   \tau_5(t)&=\nicefrac{139t^7}{5443200}+\nicefrac{1451t^5}{48988800} -\nicefrac{289517t^3}{146966400}-\nicefrac{289717t}{146966400}.
\end{align*}
The procedure \verb!Qpol! computes the polynomial $Q_k(s)$ defined in \eqref{eq:defQ}:
\begin{verbatim}
Qpol := proc(k, s)
    local T, l, h; 
    h := v -> (exp(v) - v - 1)/v^2; 
    T := 1; 
    for l to k do 
       T := diff(T, v) - s^2*T*diff(h(v), v); 
    end do; 
    limit(T, v = 0); 
end proc;
\end{verbatim}
The procedure \verb!qkj! computes $q_k^{(j)}(L)$ from the proof of Lemma \ref{lemma:asymp-varphi-ver-2} at $L$:
\begin{verbatim}
qkj := proc(j, k) 
    local T; 
    if j = 0 then 
       T := s^k*Qpol(k, s)*exp(-1/2*s^2); 
    else 
       T := diff(s^k*Qpol(k, s)*exp(-1/2*s^2), s $ j); 
    end if; 
    eval(T, s = L); 
end proc;
\end{verbatim}
The procedure \verb!compute(j,n)! computes the coefficient $a_{j,n}$ of $x^{-n/2}$ in the asymptotic expansion of $\left(\sqrt{x}\varphi_p(x)-L_p\right)^{j+1}$. 
\begin{verbatim}
compute := proc(j, n) 
    local T, q; 
    global a; 
    T := 0; 
    if j = 0 then 
       T := a[0, n - j]; 
    else for q from 0 to n - j - 1 do 
            T := T + a[j - 1, q + j]*a[0, n - j - q]; 
         end do; 
    end if; 
    T; 
end proc;
\end{verbatim}
The numbers $A_k$ can be computed using the functions
\begin{verbatim}
c := k -> int(s^k*Qpol(k, s)*exp(-1/2*s^2), s = 0 .. infinity);
d := k -> int(s^k*Qpol(k, s)*exp(-1/2*s^2), s = 0 .. L);
A := k -> ((1 - p)*(-1)^k*c(k) - p*c(k) - (-1)^k*d(k))/k!;
\end{verbatim}
For a given $M$, the coefficients $a_0,a_1,\ldots,a_{M+1}$ are computed using \eqref{eq:]}:
\begin{verbatim}
M := 7;
for j from 0 to M do
    for l from 0 to M do a[j, l] := 0; end do;
end do;
a[0, 1] := (1 + L^2/2)/3;
for N from 1 to M do
    S := 0;
    for k from 0 to N do
        l := N + 1 - k;
        for j from max(0, 1 - k) to l - 1 do
            a[j,l] := compute(j,l);
            S := S + (-1)^(k+1)*qkj(j, k)*a[j,l]/(k!*(j+1)!);
        end do;
    end do;
    S := S + A(N + 1);
    p := 1/2 - erf(L/sqrt(2))/2;
    a[0, N + 1] := simplify(S*exp(L^2/2));
end do

for l from 0 to M + 1 do
    print(eval(a[0, l]));
end do
\end{verbatim}
Finally the expansion of $m_p$ is found using $m_p(x)=xe^{-\phi_p(x)}$. We found it convenient to change variable to $y=1/\sqrt{x}$ and use \verb!taylor!:
\begin{verbatim}
phip_of_y := L*y + sum(a[0, l]*y^(l + 1), l = 0 .. M);
temp := convert(taylor(exp(-phip_of_y), y = 0, M + 1), polynom);
collect(expand(x*eval(temp, y = 1/x^(1/2))), x);
\end{verbatim}

\noindent {\bf Remark}: The author reports there are no competing interests to declare.

\noindent 
Henrik Laurberg Pedersen\\
Department of Mathematical Sciences\\
University of Copenhagen\\
Universitetsparken 5, DK-2100\\
Denmark
\end{document}